\documentclass{amsart}
\usepackage{amssymb,amstext,amsmath,amscd,amsthm,amsfonts,mathtools,stmaryrd,enumerate,graphicx,latexsym}
\usepackage[all]{xy}

\newtheorem{thm}{Theorem}[section]
\newtheorem{cor}[thm]{Corollary}
\newtheorem{lem}[thm]{Lemma}
\newtheorem{prop}[thm]{Proposition}

\theoremstyle{definition}
\newtheorem{dfn}[thm]{Definition}
\newtheorem{rem}[thm]{Remark}

\newtheorem{ex}[thm]{Example}

\theoremstyle{remark}
\newtheorem*{ac}{Acknowledgments}

\def\AA{\operatorname{A}}
\def\add{\operatorname{add}}
\def\Ass{\operatorname{Ass}}
\def\btp{\operatorname{btp}}

\def\C{\mathcal{C}}
\def\CC{\operatorname{C}}
\def\cl{\operatorname{cl}}
\def\CM{\operatorname{CM}}
\def\D{\mathcal{D}}
\def\dep{\operatorname{dep}}
\def\depth{\operatorname{depth}}
\def\dim{\operatorname{dim}}
\def\E{\mathcal{E}}
\def\Ext{\operatorname{Ext}}
\def\F{\mathcal{F}}
\def\grade{\operatorname{grade}}
\def\Hom{\operatorname{Hom}}
\def\id{\operatorname{id}}
\def\inc{\mathrm{inc}}
\def\Ker{\operatorname{Ker}}
\def\m{\mathfrak{m}}
\def\Min{\operatorname{Min}}
\def\mod{\operatorname{mod}}
\def\Mod{\operatorname{Mod}}

\def\p{\mathfrak{p}}
\def\q{\mathfrak{q}}
\def\S{\operatorname{S}}
\def\spcl{\operatorname{spcl}}
\def\Spec{\operatorname{Spec}}
\def\Supp{\operatorname{Supp}}
\def\T{\mathcal{T}}
\def\tp{\operatorname{tp}}
\def\V{\operatorname{V}}
\def\wtp{\operatorname{wtp}}
\def\X{\mathcal{X}}

\begin{document}
\allowdisplaybreaks
\title[Torsion pairs in subcategories of modules over a commutative ring]{Torsion pairs in subcategories of modules\\
over a commutative ring}
\author{Shinnosuke Kosaka}
\address{Graduate School of Mathematics, Nagoya University, Furocho, Chikusaku, Nagoya 464-8602, Japan}
\email{kosaka.shinnosuke.d2@s.mail.nagoya-u.ac.jp}
\author{Ryo Takahashi}
\address{Graduate School of Mathematics, Nagoya University, Furocho, Chikusaku, Nagoya 464-8602, Japan}
\email{takahashi@math.nagoya-u.ac.jp}
\author{Gen Tanigawa}
\address{Graduate School of Mathematics, Nagoya University, Furocho, Chikusaku, Nagoya 464-8602, Japan}
\email{tanigawa.gen.y4@s.mail.nagoya-u.ac.jp}

\begin{abstract}
Let $R$ be a commutative noetherian ring. 
Denote by $\mod R$ the category of finitely generated $R$-modules. 
Let $\Delta$ be a subset of $\Spec R$, and let $\Ass^{-1}\Delta$ stand for the full subcategory of $\mod R$ consisting of finitely generated $R$-modules whose associated prime ideals belong to $\Delta$.
In this paper, we consider classifying torsion pairs in $\Ass^{-1}\Delta$ and some other full subcategories of $\mod R$.
\end{abstract}

\maketitle

\section{Introduction}

The notion of a {\em torsion pair} was introduced by Dickson \cite{D} for abelian categories in 1966. 
A torsion pair is a pair of (full) subcategories that provides a canonical decomposition of every object into its torsion and torsionfree parts. 
In the representation theory of artin algebras, torsion theory plays an important role and has been widely investigated, particularly in tilting theory and related areas; see \cite{ASS} for instance.

In commutative algebra, classification of subcategories of modules has been actively studied.
Let $R$ be a commutative noetherian ring. 
Denote by $\mod R$ the category of finitely generated $R$-modules. 
A celebrated theorem of Gabriel \cite{G} yields a bijection between the Serre subcategories of $\mod R$ and the specialization-closed subsets of $\Spec R$. 
Takahashi \cite{T} established a one-to-one correspondence between the subcategories of $\mod R$ closed under subobjects and extensions, and the subsets of $\Spec R$.
Stanley and Wang \cite{SW} proved that a subcategory of $\mod R$ closed under extensions and cokernels is Serre.
A lot of other results on this topic have been obtained so far; see \cite{crspd,Ho,K,arg,dlr} and references therein.

In the present paper, for a commutative noetherian ring $R$, we consider classifying torsion pairs in a fixed subcategory $\X$ of $\mod R$. 
Let $(\T,\F)$ be a pair of subcategories of $\X$. 
We say that $(\T,\F)$ is a {\em weak torsion pair} in $\X$ if $\T$ consists of the modules $X\in\X$ with $\Hom_R(X,\F)=0$ and $\F$ consists of the modules $X\in\X$ with $\Hom_R(\T,X)=0$. 
When $\X$ is closed under extensions in $\mod R$, we say that $(\T,\F)$ is a {\em torsion pair} in $\X$ if $\Hom_R(\T,\F)=0$ and every module $X\in\X$ admits a short exact sequence $0\to T\to X\to F\to0$ in $\mod R$ with $T\in\T$ and $F\in\F$.
We say that $(\T,\F)$ is a {\em basic torsion pair} in $\X$ if it is a torsion pair in $\X$ admitting ideals $I,J$ of $R$ such that $\T$ consists of the modules $X\in\X$ with $IX=0$ and $\F$ consists of the modules $X\in\X$ with $JX=0$.
There are inclusions
$$
\{\text{Basic torsion pairs in $\X$}\}\subseteq
\{\text{Torsion pairs in $\X$}\}\subseteq
\{\text{Weak torsion pairs in $\X$}\},
$$
which are strict in general.
We mainly investigate torsion pairs in
$$
\Ass^{-1}\Delta=\{M\in\mod R\mid\Ass M\subseteq\Delta\}
$$
for a fixed subset $\Delta$ of $\Spec R$.
The following theorem is the main result of this paper.

\begin{thm}[Theorems \ref{thm tp}, \ref{thm btp}, and \ref{thm btp min}]\label{1}
Let $R$ be a commutative noetherian ring.
Let $\Delta$ be a subset of $\Spec R$. 
Then there is a commutative diagram
$$
\xymatrix{
{\left\{\begin{matrix}\text{Weak torsion}\\
\text{pairs in $\Ass^{-1}\Delta$}\end{matrix}\right\}}\ar@{=}[r]&
{\left\{\begin{matrix}\text{Torsion pairs}\\
\text{in $\Ass^{-1}\Delta$}\end{matrix}\right\}}
\ar[r]^-{\sigma}_-\cong&
{\left\{\begin{matrix}\text{Specialization-closed}\\
\text{subsets of $\Delta$}\end{matrix}\right\}}
\\
&{\left\{\begin{matrix}\text{Basic torsion}\\
\text{pairs in $\Ass^{-1}\Delta$}\end{matrix}\right\}}
\ar[u]^-\inc\ar[r]^-\tau
&
{\left\{\begin{matrix}\text{Closed subsets $C$}\\
\text{of $\Delta$ such that}\\
\text{$C\sqcup C'=\Delta$ for some}\\
\text{closed subset $C'$ of $\Delta$}
\end{matrix}\right\}}
\ar[u]_-\inc
}
$$
of maps of sets, where $\inc$ means an inclusion map, $\sigma$ is a bijection, and so is $\tau$ if $\Delta$ contains $\Min R$.
\end{thm}

\noindent
The maps $\sigma$ and $\tau$ can be explicitly given.
Applying the theorem for $\Delta=\Spec R$ and $\Delta=\Min R$, we obtain complete classifications of the weak torsion pairs, the torsion pairs, and the basic torsion pairs in $\mod R$ and the subcategory $\S_1(R)$ of $\mod R$ consisting of $R$-modules satisfying Serre's condition $(\S_1)$.

\begin{cor}[Corollaries \ref{cor tp modR} and \ref{cor S1}]
Let $R$ be a commutative noetherian ring. 
One has the following equalities and bijections.
\begin{align*}
\left\{\begin{matrix}\text{Weak torsion}\\
\text{pairs in $\mod R$}\end{matrix}\right\}
&=\left\{\begin{matrix}\text{Torsion pairs}\\
\text{in $\mod R$}\end{matrix}\right\}
\cong\left\{\begin{matrix}\text{Specialization-closed}\\
\text{subsets of $\Spec R$}\end{matrix}\right\},
\\
\left\{\begin{matrix}\text{Basic torsion}\\
\text{pairs in $\mod R$}\end{matrix}\right\}
&\cong\left\{(C,C')\,\bigg|\,\begin{matrix}\text{$C,C'$ are closed subsets of}\\
\text{$\Spec R$ with $C\sqcup C'=\Spec R$}\end{matrix}\right\},\\
\left\{\begin{matrix}\text{Weak torsion}\\
\text{pairs in $\S_1(R)$}\end{matrix}\right\}
&=\left\{\begin{matrix}\text{Torsion pairs}\\
\text{in $\S_1(R)$}\end{matrix}\right\}
=\left\{\begin{matrix}\text{Basic torsion}\\
\text{pairs in $\S_1(R)$}\end{matrix}\right\}
\cong\left\{\begin{matrix}\text{Subsets}\\
\text{of $\Min R$}\end{matrix}\right\}.
\end{align*}
\end{cor}

We also study weak torsion pairs in subcategories of $\mod R$ that are not of the form $\Ass^{-1}\Delta$, such as the subcategory $\CM(R)$ of maximal Cohen--Macaulay modules over a Cohen--Macaulay local ring $R$.
For a subcategory $\X$ of $\mod R$, we set $\Ass\X=\bigcup_{M\in\X}\Ass M$.
The following theorem holds.

\begin{thm}[Theorem \ref{thm A}]\label{3}
Let $R$ be a commutative noetherian ring and $\X$ a subcategory of $\mod R$. Suppose that for every $\p\in\Ass\X$, there exists $X\in\X$ such that $\Ass X=\{\p\}$.
One then has a bijection
$$
\left\{\begin{matrix}\text{Weak torsion}\\
\text{pairs in $\X$}\end{matrix}\right\}
\cong
\left\{\begin{matrix}\text{Specialization-closed}\\
\text{subsets of $\Ass\X$}\end{matrix}\right\}.
$$
\end{thm}

Applying this theorem, we can classify the weak torsion pairs in $\CM(R)$ when $R$ is a local hypersurface.

\begin{cor}[Corollary \ref{cor HS}]
Let $R=S/(f)$ be a hypersurface local ring, where $S$ is a regular local ring and $f$ is a nonunit element of $S$. Then there exists a one-to-one correspondence
$$
\left\{\begin{matrix}\text{Weak torsion}\\
\text{pairs in $\CM(R)$}\end{matrix}\right\}
\cong
\left\{\begin{matrix}\text{Subsets}\\
\text{of $\Min R$}\end{matrix}\right\}.
$$
\end{cor}

We prove Theorems \ref{1} and \ref{3} in Section 4, after giving basic notions and their fundamental properties in Sections 2 and 3. In Section 5, we apply the results in Section 4 to obtain various classifications of torsion pairs in a fixed subcategory of $\mod R$.
Throughout this paper, unless otherwise specified, we assume that all rings are commutative noetherian ring with identity, all modules are finitely generated, and all subcategories are strictly full. For simplicity, subscripts are omitted when no confusion arises.

\section{Weak torsion pairs in preadditive categories}

Throughout this section, let $\C$ be a preadditive category. For a subcategory $\D$ of $\C$, we denote by $\D^\perp$ (resp. $^\perp\D$) the right (resp. left) orthogonal subcategory of $\C$, that is, the subcategory of $\C$ consisting of objects $C\in\C$ such that $\Hom_\C(D,C)=0$ (resp. $\Hom_\C(C,D)=0$) for all $D\in\D$. We begin with the definition of a weak torsion pair.

\begin{dfn}
A pair $(\T,\F)$ of subcategories of $\C$ is called a {\em weak torsion pair} in $\C$ if $\T^\perp=\F$ and $^\perp\F=\T$. In this case, $\T$ (resp. $\F$) is called the {\em torsion part} (resp. {\em torsionfree part}) of $(\T,\F)$. We denote by $\wtp(\C)$ the collection of weak torsion pairs in $\C$.
\end{dfn}

The orthogonal subcategory of a subcategory of $\C$ has the following closedness property. For a sequence $A\overset{f}\to B\overset{g}\to C$ in $\C$, we say that $f$ is a {\em pseudo-kernel} of $g$ (resp. $g$ is a {\em pseudo-cokernel} of $f$) if the sequence $\Hom_\C(X,A)\overset{f\circ-}\to\Hom_\C(X,B)\overset{g\circ-}\to\Hom_\C(X,C)$ (resp. $\Hom_\C(C,X)\overset{-\circ g}\to\Hom_\C(B,X)\overset{-\circ f}\to\Hom_\C(A,X)$) of abelian groups is exact for all $X\in\C$.

\begin{lem}\label{lem closed under}
Let $\D$ be a subcategory of $\C$.
\begin{enumerate}[\rm(1)]
\item 
The subcategory $\D^\perp$ of $\C$ is closed under subobjects.
\item 
The subcategory $^\perp\D$ of $\C$ is closed under quotient objects.
\item 
Let $A\overset{f}\to B\overset{g}\to C$ be a sequence of morphisms in $\C$ such that $f$ is a pseudo-kernel of $g$. If $A,C\in\D^\perp$, then $B\in\D^\perp$.
\item 
Let $A\overset{f}\to B\overset{g}\to C$ be a sequence of morphisms in $\C$ such that $g$ is a pseudo-cokernel of $f$. If $A,C\in{}^\perp\D$, then $B\in{}^\perp\D$.
\item 
The subcategories $\D^\perp$ and $^\perp\D$ of $\C$ are closed under direct sums and direct summands.
\end{enumerate}
\end{lem}

\begin{proof}
The proofs of (2) and (4) are similar to those of (1) and (3), and (5) follows immediately from (1) and (3). Therefore, we only prove (1) and (3).

(1) Let $f:A\to B$ be a monomorphism in $\C$ with $B\in\D^\perp$. For all $D\in\D$, the sequence $0\to\Hom_\C(D,A)\to\Hom_\C(D,B)$ is exact. We get $\Hom_\C(D,A)=0$ as $\Hom_\C(D,B)=0$. We have $A\in\D^\perp$.

(3) For all $D\in\D$, the sequence $\Hom_\C(D,A)\to\Hom_\C(D,B)\to\Hom_\C(D,C)$ is exact. We see that $\Hom_\C(D,B)=0$ since $\Hom_\C(D,A)=0$ and $\Hom_\C(D,C)=0$. Thus, we have $B\in\D^\perp$.
\end{proof}

For later use, we record the following elementary fact.

\begin{lem}\label{formula perp}
Let $\D$ be a subcategory of $\C$.
\begin{enumerate}[\rm(1)]
\item 
The equality $^\perp((^\perp\D)^\perp)={}^\perp\D$ holds.
\item 
The equality $(^\perp(\D^\perp))^\perp=\D^\perp$ holds.
\item 
The following conditions are equivalent.
\begin{enumerate}[\rm(a)]
\item 
$\D$ is a torsion part of some weak torsion pair in $\C$.
\item 
There exists a subcategory $\D'$ such that $\D={}^\perp\D'$.
\item 
The equality $^\perp(\D^\perp)=\D$ holds.
\end{enumerate}
\item 
The following conditions are equivalent.
\begin{enumerate}[\rm(a)]
\item 
$\D$ is a torsionfree part of some weak torsion pair in $\C$.
\item 
There exists a subcategory $\D'$ such that $\D=\D'^\perp$.
\item 
The equality $(^\perp\D)^\perp=\D$ holds.
\end{enumerate}
\end{enumerate}
\end{lem}

\begin{proof}
The proofs of (2) and (4) are similar to those of (1) and (3). Therefore, we only prove (1) and (3).

(1) The inclusion $^\perp\D\subseteq{}^\perp((^\perp\D)^\perp)$ is clear. The reverse inclusion follows immediately from $\D\subseteq(^\perp\D)^\perp$.

(3) The implications (c)$\Rightarrow$(a) and (a)$\Rightarrow$(b) are clear. The implication (b)$\Rightarrow$(c) follows from (1).
\end{proof}

Clearly, there exists no nonzero object that is both "torsion" and "torsionfree".

\begin{lem}\label{lem intersection}
Let $(\T,\F)$ be a weak torsion pair in $\C$. Then every object in $\T\cap\F$ is a zero object.
\end{lem}

\begin{proof}
Let $X$ be an object of $\T\cap\F$. It follows from the definition that $\Hom_\C(X,X)=0$. Therefore we have $\id_X=0$, which implies $X=0$.
\end{proof}

In the rest of this section, let $\X$ be a subcategory of $\C$. For a subcategory of $\D$, we set $\D^\perp_\X\coloneq\D^\perp\cap\X$ and $^\perp_\X\D\coloneq{}^\perp\D\cap\X$. The following lemma is an analogue of Lemma \ref{formula perp}.

\begin{lem}\label{formula perp X}
Let $\D$ be a subcategory of $\C$.
\begin{enumerate}[\rm(1)]
\item 
The equality $^\perp_\X((^\perp_\X\D)^\perp)={}^\perp_\X\D$ holds.
\item 
The equality $(^\perp(\D^\perp_\X))^\perp_\X=\D^\perp_\X$ holds.
\end{enumerate}
\end{lem}

\begin{proof}
The proof of (2) is similar to that of (1). Therefore, we only prove (1).

The inclusion $\D^\perp_\X\subseteq(^\perp(\D^\perp_\X))^\perp_\X$ is clear. By definition, we see that $(^\perp(\D^\perp_\X))^\perp\cap\X$ is contained in $(^\perp(\D^\perp))^\perp\cap\X$. It follows from Lemma \ref{formula perp} (1) that the equation $(^\perp(\D^\perp))^\perp\cap\X=\D^\perp\cap\X$ holds. Thus, the desired conclusion follows.
\end{proof}

Now we consider a correspondence between weak torsion pairs in $\X$ and weak torsion pairs in $\C$.

\begin{dfn}\label{dfn fg}
Let $(\T,\F)$ be a pair of subcategories of $\C$.
\begin{enumerate}[\rm(1)]
\item 
We denote the pair $(\T\cap\X,\F\cap\X)$ of subcategories of $\X$ by $(\T,\F)\cap\X$.
\item 
We denote the weak torsion pair $(^\perp(\T^\perp),\T^\perp)$ (resp. $(^\perp\F,(^\perp\F)^\perp)$) in $\C$ by $f_\C(\T,\F)$ (resp. $g_\C(\T,\F)$).
\item 
We denote the weak torsion pair $f_\X((\T,\F)\cap\X)=(^\perp_\X((\T\cap\X)^\perp_\X),(\T\cap\X)^\perp_\X)$ (resp. $g_\X((\T,\F)\cap\X)=(^\perp_\X(\F\cap\X),(^\perp_\X(\F\cap\X))^\perp_\X)$) in $\X$ by $f^\X_\C(\T,\F)$ (resp. $g^\X_\C(\T,\F)$).
\end{enumerate}
\end{dfn}

\begin{rem}\label{rem fg}
Let $(\T,\F)$ be a pair of subcategories of $\C$. If $(\T,\F)$ is a weak torsion pair in $\C$, then one has $f_\C(\T,\F)=g_\C(\T,\F)=(\T,\F)$ by (3) and (4) of Lemma \ref{formula perp}. Similarly, if $(\T,\F)\cap\X$ is a weak torsion pair in $\X$, then one has $f^\X_\C(\T,\F)=g^\X_\C(\T,\F)=(\T,\F)\cap\X$.
\end{rem}

The following proposition states an important property of the notation introduced above.

\begin{prop}\label{prop wtp subcat}
Let $(\T,\F)$ be a weak torsion pair in $\X$, $\alpha\in\{f_\C,g_\C\}$, and $\beta\in\{f^\X_\C,g^\X_\C\}$. Then one has $\beta(\alpha(\T,\F))=\alpha(\T,\F)\cap\X=(\T,\F)$.
\end{prop}

\begin{proof}
We only prove the case $\alpha=f_\C$. The proof of the case $\alpha=g_\C$ is similar.

By Lemma \ref{formula perp X}, we see that $f_\C(\T,\F)\cap\X=(^\perp(\T^\perp)\cap\X,\T^\perp\cap\X)=(^\perp_\X((^\perp_\X\F)^\perp),\T^\perp_\X)=(^\perp_\X\F,\T^\perp_\X)=(\T,\F)$. Hence, the equality $f_\C (\T,\F)\cap \X=(\T,\F)$ holds. In particular, $f_\C(\T,\F)\cap\X$ is a torsion pair in $\X$. Thus, the equality $\beta(f_\C(\T,\F))=f_\C(\T,\F)\cap\X$ follows from Remark \ref{rem fg}.
\end{proof}

From the above proposition, we observe that every weak torsion pair in $\X$ is obtained as the intersection of $\X$ with some weak torsion pair in $\C$.

\begin{cor}\label{cor wtp restrict}
Let $(\T,\F)$ be a weak torsion pair in $\X$, then there exists a torsion pair $(\T',\F')$ in $\C$ such that $(\T',\F')\cap\X$ coincides with $(\T,\F)$.
\end{cor}

\section{Torsion pairs in exact categories}

Throughout this section, let $\E$ be an exact category in the sense of Quillen \cite{Q}. We first define a torsion pair in $\E$.

\begin{dfn}
Let $(\T,\F)$ be a pair of subcategories in $\E$. 
We say that $(\T,\F)$ is a {\em torsion pair} in $\E$ if it satisfies the following two conditions.
\begin{enumerate}[\rm(1)]
\item 
The equality $\Hom_\E(T,F)=0$ holds for all objects $T\in\T,F\in\F$.
\item
For all objects $E\in\E$, there exists a conflation $0\to T\to E\to F\to0$ with $T\in\T$ and $F\in\F$. This sequence is called a canonical sequence of $E$ with respect to $(\T,\F)$.
\end{enumerate}
We denote by $\tp(\E)$ the collection of torsion pairs in $\E$.
\end{dfn}

The following lemma shows that the terminology of a "weak" torsion pair is coherent.

\begin{lem}
If $(\T,\F)$ is a torsion pair in $\E$, then $(\T,\F)$ is a weak torsion pair in $\E$.
\end{lem}

\begin{proof}
It is straightforward to see that $\F$ (\emph{resp. $\T$}) is contained in $\T^\perp$ (\emph{resp. ${}^\perp\F$}). 
Let $E$ be an object in $\T^\perp$. Then there exists a conflation $0\to T\to E\to F\to0$ with $T\in \T$ and $F\in \F$. Since $\Hom_\E(T,E)=0$, we have $E\cong F$. Therefore, $E$ belongs to $\F$, which implies $\T^\perp\subseteq\F$. Similarly, we have $^\perp\F\subseteq\T$. Hence, we conclude that $(\T,\F)$ is a weak torsion pair.
\end{proof}

By Lemma \ref{lem closed under}, we see that the torsion part and the torsionfree part of a weak torsion pair in $\E$ have the following closedness property.

\begin{cor}
Let $(\T,\F)$ be a weak torsion pair in $\E$. Then $\T$ \emph{(resp. $\F$)} is closed under extensions and quotient objects \emph{(resp.} extensions and subobjects\emph{)}.
\end{cor}

The following proposition says that a canonical sequence of a torsion pair $(\T,\F)$ in an abelian category is unique up to isomorphisms.

\begin{prop}
Suppose that $\E$ is an abelian category. Let $(\T,\F)$ be a torsion pair in $\E$, and $E$ an object in $\E$. If there exist two exact sequences $0\to T_1\to E\to F_1\to 0$ and $0\to T_2\to E\to F_2\to 0$ with $T_1,T_2\in \T$ and $F_1,F_2\in \F$, then one has $T_1\cong T_2$ and $F_1\cong F_2$.
\end{prop}

\begin{proof}
We consider the following pullback diagram.
$$
\begin{CD}
 @. 0 @. 0 @. 0 @.\\
@. @VVV @VVV @VVV @.\\
0 @>>> X_4 @>>> T_2 @>>> X_3 @>>> 0\\
@. @VVV @VVV @VVV @.\\
0 @>>> T_1 @>>> E @>>> F_1 @>>> 0\\
@. @VVV @VVV @VVV @.\\
0 @>>> X_2 @>>> F_2 @>>> X_1 @>>> 0\\
@. @VVV @VVV @VVV @.\\
 @. 0 @. 0 @. 0 @.
\end{CD}
$$
The object $X_2$ is a quotient object of $T_1$ and a subobject of $F_2$, so that we have $X_2=0$. Similarly, we obtain $X_3=0$. This yields $T_1\cong X_4 \cong T_2$ and $F_1\cong X_1\cong F_2$.
\end{proof}

Let $(\T,\F)$ be a torsion pair in $\E$ and $E$ an indecomposable object in $E$. Obviously, the canonical sequence of $E$ with respect to $(\T,\F)$ splits if and only if $E$ is contained in either $\T$ or $\F$. The following proposition shows when every canonical sequence in $\E$ with respect to $(\T,\F)$ splits.

\begin{prop}\label{prop cano seq split}
Suppose $\E$ is an abelian category. If both $(\T,\F)$ and $(\F,\T)$ are torsion pairs in $\E$, then all canonical sequences split. In particular, all indecomposable objects in $\E$ belong to either $\T$ or $\F$.
\end{prop}

\begin{proof}
By Lemma \ref{lem closed under}, the subcategories $\T,\F$ of $\E$ are Serre. For all $M \in \E$, there exist two exact sequences $0 \to T_M \to M \to {}_MF \to 0$ and $0\to F_M \to M \to {}_MT\to 0$ with $T_M,{}_MT \in \T$ and $F_M,{}_MF \in \F$. There exists a pullback diagram.
$$
\begin{CD}
 @. 0 @. 0 @. 0 @.\\
@. @VVV @VVV @VVV @.\\
0 @>>> X_4 @>>> T_M @>>> X_3 @>>> 0\\
@. @VVV @VVV @VVV @.\\
0 @>>> F_M @>>> M @>>> {}_MT @>>> 0\\
@. @VVV @VVV @VVV @.\\
0 @>>> X_2 @>>> {}_MF @>>> X_1 @>>> 0\\
@. @VVV @VVV @VVV @.\\
 @. 0 @. 0 @. 0 @.
\end{CD}
$$

The object $X_4$ is a subobject of $T_M$ and $F_M$, so $X_4$ belongs to $\T$ and $\F$. Similarly, $X_1$ is a quotient object of ${}_MT$ and ${}_MF$, so $X_1$ belongs to $\T$ and $\F$. Hence $X_1=X_4=0$. We obtain ${}_MF \cong X_2 \cong F_M$ and ${}_MT \cong X_3 \cong T_M$. Therefore the canonical sequences split. 
\end{proof}

We recall the notion of a cotorsion pair. For a subcategory $\D$ of $\E$, denote by $\D^{\perp_1}$ (resp. $^{\perp_1}\D$) the subcategory of $\E$ consisting of objects $E\in\E$ with $\Ext^1_\C(D,E)=0$ (resp. $\Ext^1_\C(E,D)=0$) for all $D\in\D$.

\begin{dfn}\label{dfn cotorsion pair}
Let $(\T,\F)$ be a pair of subcategories of $\E$.
\begin{enumerate}[\rm(1)]
\item 
We say that $(\T,\F)$ is a {\em weak cotorsion pair} in $\E$ if $\T^{\perp_1}=\F$ and $^{\perp_1}\F=\T$.
\item 
We say that $(\T,\F)$ is a {\em cotorsion pair} in $\E$ if it satisfies the following three conditions.
\begin{enumerate}[\rm(a)]
\item 
The subcategories $\T$ and $\F$ of $\E$ are closed under direct summands.
\item 
The equality $\Ext^1_\C(T,F)=0$ holds for all objects $T\in\T,F\in\F$.
\item 
For all objects $E\in\E$, there exist conflations $0\to F\to T\to E\to0$ and $0\to E\to F'\to T'\to0$ with $T,T'\in\T$ and $F,F'\in\F$.
\end{enumerate}
\end{enumerate}
\end{dfn}

As in the case of a torsion pair, the terminology of a ``weak" cotorsion pair is coherent.

\begin{lem}
If $(\T,\F)$ is a cotorsion pair in $\E$, then $(\T,\F)$ is a weak cotorsion pair in $\E$.
\end{lem}

\begin{proof}
It is straightforward to see that $\F$ (\emph{resp. $\T$}) is contained in $\T^{\perp_1}$ (\emph{resp. ${}^{\perp_1}\F$}). Let $E$ be an object in $\T^{\perp_1}$. Then there exists a conflation $0\to E\to F\to T\to0$ such that $T\in \T$ and $F\in \F$. It follows from $\Ext^1_\C(T,E)=0$ that the sequence $0\to E\to F\to T\to0$ splits, which implies that $E$ is a direct summand of $T$. Since $\T$ is closed under direct summand in $\E$, $E$ belongs to $\T$. Therefore, we have $\T^{\perp_1}\subseteq\F$. Similarly, we see that $^{\perp_1}\F\subseteq\T$. Hence, we conclude that $(\T,\F)$ is a weak cotorsion pair.
\end{proof}

For a subcategory $\D$ of $\E$, we denote by $\add\D$ the subcategory of $\E$ consisting of direct summands of finite direct sums of objects in $\D$. For a module $M$ over a local ring $R$, we denote by $\Omega_RM$ the first syzygy module of $M$ in its minimal $R$-free resolution. For a subcategory $\X$ of $\mod R$, we denote by $\Omega_R\X$ the subcategory of $\mod R$ consisting of $R$-modules $M$ such that there exists an exact sequence $0\to M\to F\to X\to0$ of $R$-modules where $X\in\X$ and $F$ is a free $R$-module. Note that an $R$-module $M$ belongs to $\Omega_R\X$ if and only if there exist an $R$-module $X\in\X$ and a free $R$-module $F$ such that $M\cong\Omega_RX\oplus F$. When $R$ is Cohen--Macaulay, we denote by $\CM(R)$ the subcategory of $\mod R$ consisting of maximal Cohen--Macaulay $R$-modules. The following proposition shows that a torsion pair in $\CM(R)$ induces a cotorsion pair in $\CM(R)$ if $R$ is Gorenstein.

\begin{prop}
Let $R$ be a Gorenstein local ring. If $(\T,\F)$ is a torsion pair in $\CM(R)$, then the pair $(\add(\T\cup\{R\}),\add(\Omega\F))$ is a cotorsion pair in $\CM(R)$.
\end{prop}

\begin{proof}
We prove that the pair $(\add(\T\cup\{R\}),\add(\Omega\F))$ of subcategories of $\CM(R)$ satisfies conditions (a), (b), and (c) of Definition \ref{dfn cotorsion pair} (2).

(a) Clear.

(b) If $T\in\T$ and $F\in\F$, then we have $\Ext^1_R(T,\Omega_RF)\cong\underline\Hom_R(\Omega_RT,\Omega_RF)\cong\underline\Hom_R(T,F)=0$ since $\Hom_R(T,F)=0$. Here, for $R$-modules $M$ and $N$, we set $\underline\Hom_R(M,N)=\Hom_R(M,N)/P(M,N)$, where $P(M,N)$ denotes the submodule of $\Hom_R(M,N)$ consisting of $R$-homomorphisms from $M$ to $N$ which factor through a $R$-free module.

(c) Let $M$ be an $R$-module belonging to $\CM(R)$. Then there exists a short exact sequence $0\to L\to M\to N\to0$ such that $L\in\T$ and $N\in\F$. Then there is a pullback diagram
\[\xymatrix{
&&0\ar[d]&0\ar[d]&\\
&&\Omega_RN\ar@{=}[r]\ar[d]&\Omega_RN\ar[d]&\\
0\ar[r]&L\ar[r]\ar@{=}[d]&X\ar[r]\ar[d]&F\ar[r]\ar[d]&0\\
0\ar[r]&L\ar[r]&M\ar[r]\ar[d]&N\ar[r]\ar[d]&0\\
&&0&0&
}\]
where $F$ is a free $R$-module. Since the middle row in the above diagram splits, we have $X\cong L\oplus F\in\add(\T\cup\{R\})$. Therefore, we obtain a short exact sequence $0\to \Omega_RN\to X\to M\to0$ with $X\in\add(\T\cup\{R\})$ and $\Omega_RN\in\add(\Omega\F)$.

Moreover, since $R$ is Gorenstein, there is an exact sequence $0\to M\to G\to Y\to0$ such that $G$ is $R$-free and $Y\in\CM(R)$. Repeating the above argument, there is a short exact sequence $0\to Z\to W\to Y\to0$ such that $W\in\add(\T\cup\{R\})$ and $Z\in\add(\Omega\F)$. We consider the following pullback diagram.
\[\xymatrix{
&&0\ar[d]&0\ar[d]&\\
&&M\ar@{=}[r]\ar[d]&M\ar[d]&\\
0\ar[r]&Z\ar[r]\ar@{=}[d]&U\ar[r]\ar[d]&G\ar[r]\ar[d]&0\\
0\ar[r]&Z\ar[r]&W\ar[r]\ar[d]&Y\ar[r]\ar[d]&0\\
&&0&0&
}\]
As the middle row in the above diagram splits, we have $U\cong Z\oplus G\in\add(\Omega\F)$. We obtain a short exact sequence $0\to M\to U\to W\to0$ with $U\in\add(\Omega_R\F)$ and $W\in\add(\T\cup\{R\})$.
\end{proof}

\section{Torsion pairs in module categories}

Throughout this section, let $R$ be a ring. We investigate torsion pairs in a fixed subcategory of $\mod R$ and establish our main theorems. For a subcategory $\D$ of $\mod R$, the notation $\D^\perp$ (resp. $^\perp\D$) refers to the right (resp. left) orthogonal subcategory in $\mod R$. For a topological space $X$, a subset $W$ of $X$ is called {\em specialization-closed} if the closure of $\{w\}$ in $X$ is contained in $W$ for all $w\in W$. We denote by $\spcl X$ (resp. $\cl X$) by the set of all specialization-closed subsets (resp. closed subsets) of $X$. A subset $W$ of $\Spec R$ is specialization-closed with respect to the
Zariski topology if and only if for all prime ideals $\p\in W$ and $\q\in\Spec R$
with $\p\subseteq\q$, we have $\q\in W$. We begin with the definition and basic properties of $W$-torsion and $W$-torsionfree modules for a specialization-closed subset $W$ of $\Spec R$.

\begin{dfn}
Denote by $\Mod R$ the category of all (possibly infinitely generated) $R$-modules.
Let $W$ be a specialization-closed subset of $\Spec R$, and let $M\in\Mod R$.
\begin{enumerate}[\rm(1)]
\item
We denote by $\Gamma_W(M)$ the submodule $\{x\in M\mid\Supp(Rx)\subseteq W\}$ of $M$.
\item
We say that $M$ is {\em $W$-torsion} if $\Gamma_W(M)=M$.
\item 
We say that $M$ is {\em $W$-torsionfree} if $\Gamma_W(M)=0$.
\end{enumerate}
\end{dfn}

\begin{rem}\label{rem gamma}
Let $W$ be a specialization-closed subset of $\Spec R$ and $M\in\Mod R$.
\begin{enumerate}[\rm(1)]
\item
The $R$-module $M$ is $W$-torsion if and only if $\Supp M\subseteq W$.
\item
The $R$-module $M$ is $W$-torsionfree if and only if $\Ass M\cap W=\emptyset$.
\item
The $R$-module $\Gamma_W(M)$ is $W$-torsion.
\item
The $R$-module $M/\Gamma_W(M)$ is $W$-torsionfree.
\item 
The functor $\Gamma_W(-):\Mod R\to\Mod R$ is well-defined and left exact.
\end{enumerate}
\end{rem}

We prove the following lemma, which will be used in the proof of Lemma \ref{formula gamma ass}.

\begin{lem}\label{lem ex seq torsion}
Let $W$ be a specialization-closed subset of $\Spec R$.
Let $0\to L\to M\to N$ be a short exact sequence in $\Mod R$.
If $L$ is $W$-torsion and $N$ is $W$-torsionfree, then the induced homomorphism $L\to \Gamma_W(M)$ is an isomorphism.
\end{lem}

\begin{proof}
By the left exactness of $\Gamma_W(-)$, the sequence $0\to \Gamma_W(L)\to \Gamma_W(M)\to\Gamma_W(N)$ is exact. The assertion follows since $\Gamma_W(L)=L$ and $\Gamma_W(N)=0$.
\end{proof}

The following lemma plays an important role for the proof of the main theorem.

\begin{lem}\label{formula gamma ass}
Let $W$ be a specialization-closed subset of $\Spec R$.
Let $M\in\Mod R$.
Then $\Ass(\Gamma_W(M))=\Ass M\cap W$ and $\Ass(M/\Gamma_W(M))=\Ass M\setminus W$.
\end{lem}

\begin{proof}
The exact sequence $0\to\Gamma_W(M)\to M\to M/\Gamma_W(M)\to0$ implies that $\Ass(\Gamma_W(M))\subseteq\Ass M\subseteq\Ass(\Gamma_W(M))\cup\Ass(M/\Gamma_W(M))$. Hence, by Remark \ref{rem gamma}, it is enough to show that $\Ass(M/\Gamma_W(M))\subseteq\Ass M$. Suppose that there exists a prime ideal $\p\in\Ass(M/\Gamma_W(M))\setminus\Ass M$. Then there is an injective homomorphism $R/\p\to M/\Gamma_W(M)$. Consider the following pullback diagram.
\[\xymatrix{&&0\ar[d]&0\ar[d]&\\
0\ar[r]&\Gamma_W(M)\ar[r]\ar@{=}[d]&X\ar[r]\ar[d]&R/\p\ar[r]\ar[d]&0\\
0\ar[r]&\Gamma_W(M)\ar[r]&M\ar[r]&M/\Gamma_W(M)\ar[r]&0}\]

We see that $\Ass(\Gamma_W(M))\subseteq\Ass X\subseteq\Ass(\Gamma_W(M))\cup\Ass(R/\p)=\Ass(\Gamma_W(M))\cup\{\p\}$. On the other hand, we have $\p\notin\Ass(X)$ since $\Ass X\subseteq\Ass M$. Hence, we obtain $\Ass X=\Ass(\Gamma_W(M))\subseteq W$, so that $X$ is $W$-torsion. By Lemma \ref{lem ex seq torsion}, the homomorphism $\Gamma_W(M)\to X$ is an isomorphism and therefore $R/\p=0$. This is a contradiction. We get $\Ass(M/\Gamma_W(M))\subseteq\Ass M$ and the assertion follows.
\end{proof}

Next, we define the following correspondence between subcategories of $\mod R$ and subsets of $\Spec R$.

\begin{dfn}
Let $\X$ be a subcategory of $\mod R$ and $W$ a subset of $\Spec R$. We set
\begin{enumerate}[\rm(1)]
\item 
$\Supp\X\coloneq\bigcup_{X\in\X}\Supp X$,
\item 
$\Ass\X\coloneq\bigcup_{X\in\X}\Ass X$,
\item 
$\Supp^{-1}W\coloneq\{M\in\mod R\mid\Supp M\subseteq W\}$, and
\item 
$\Ass^{-1}W\coloneq\{M\in\mod R\mid\Ass M\subseteq W\}$.
\end{enumerate}
\end{dfn}

\begin{rem}\label{rem T}
According to \cite[Theorem 4.1]{T}, there is the following commutative diagram, where the maps in each row are mutually inverse bijections.
\[\xymatrix{{\left\{\begin{array}{c}
\text{Subcategories of $\mod R$ closed under}\\
\text{subobjects and extensions}
\end{array}\right\}}\ar@<0.5ex>[r]^-\Ass&2^{\Spec R}\ar@<0.5ex>[l]^-{\Ass^{-1}}\\
\{\text{Serre subcategories of $\mod R$}\}\ar[u]_-\inc\ar@<0.5ex>[r]^-\Supp&\spcl(\Spec R)\ar[u]_-\inc\ar@<0.5ex>[l]^-{\Supp^{-1}}}\]
Here, $\inc$ means an inclusion map, and we denote the power set of a set $X$ by $2^X$.
\end{rem}

Using the above notation, the orthogonal subcategory of a subcategory of $\mod R$ is described as follows. For a subset $W$ of $\Spec R$, we write $W^\complement =\Spec R\setminus W$.

\begin{lem}\label{formula perp modR}
Let $\D$ be a subcategory of $\mod R$.
\begin{enumerate}[\rm(1)]
\item
The equality $\D^\perp=\Ass^{-1}((\Supp\D)^\complement )$ holds.
\item 
The equality $^\perp\D=\Supp^{-1}((\Ass\D)^\complement )$ holds.
\end{enumerate}
\end{lem}

\begin{proof}
The proof of (2) is similar to that of (1). Therefore, we only prove (1).

Let $M$ be an $R$-module. Then the following equivalences hold:
\[\begin{array}{ccl}
M\in\D^\perp&\iff&\text{$\Hom_R(D,M)=0$ for all $D\in\D$}\\
&\iff&\text{$\Supp D\cap\Ass M=\emptyset$ for all $D\in\D$}\\
&\iff&\Supp\D\cap\Ass M=\emptyset\\
&\iff&M\in\Ass^{-1}((\Supp\D)^\complement ).
\end{array}\]
Therefore, we have $\D^\perp=\Ass^{-1}((\Supp\D)^\complement )$.
\end{proof}

In the rest of this section, let $\Delta$ be a subset of $\Spec R$. One of our aims is to prove Theorem \ref{1}. For a subcategory $\D$ of $\mod R$ and a subcategory $\X$ of $\Ass^{-1}\Delta$, we simply write $\D^\perp_\Delta$ (resp. $^\perp_\Delta\D$, $f_\Delta$, $f^\X_\Delta$) for $\D^\perp_{\Ass^{-1}\Delta}$ (resp. $^\perp_{\Ass^{-1}\Delta}\D$, $f_{\Ass^{-1}\Delta}$, $f^\X_{\Ass^{-1}\Delta}$); see Definition \ref{dfn fg} for the definitions of $f_{\Ass^{-1}\Delta}$ and $f^\X_{\Ass^{-1}\Delta}$. The subcategories $\D^\perp_\Delta$ and $^\perp_\Delta(\D^\perp_\Delta)$ are described as follows.

\begin{lem}\label{formula perp delta}
Let $\D$ be a subcategory of $\mod R$.
\begin{enumerate}[\rm(1)]
\item 
The equality $\D^\perp_\Delta=\Ass^{-1}(\Delta\setminus\Supp\D)$ holds.
\item 
The equality $^\perp_\Delta(\D^\perp_\Delta)=\Ass^{-1}(\Delta\cap\Supp\D)$ holds.
\end{enumerate}
In particular, for a pair $(\T,\F)$ of subcategories of $\Ass^{-1}\Delta$, one has $f_\Delta(\T,\F)=(\Ass^{-1}(\Delta\cap\Supp\T),\Ass^{-1}(\Delta\setminus\Supp\T))$.
\end{lem}

\begin{proof}
(1) This is clear from Lemma \ref{formula perp modR} (1).

(2) Let $W=\Supp\D$. It follows from (1) and Remark \ref{rem T} that $\Ass(\D^\perp_\Delta)=\Ass(\Ass^{-1}(\Delta\setminus W))=\Delta\setminus W$. Therefore, by Lemma \ref{formula perp modR} (2), we have $^\perp_\Delta(\D^\perp_\Delta)=\Supp^{-1}((\Delta\setminus W)^\complement )\cap\Ass^{-1}\Delta=\Supp^{-1}(\Delta^\complement \cup W)\cap\Ass^{-1}\Delta=\Ass^{-1}(\Delta\cap W)$.
\end{proof}

We are now ready to give a classification of torsion pairs in $\Ass^{-1}\Delta$. The statement is as follows. We note that a subset $Z$ is specialization-closed in $\Delta$ if and only if there exists a specialization-closed subset $W$ of $\Spec R$ such that $Z=\Delta\cap W$.

\begin{thm}\label{thm tp}
Every weak torsion pair in $\Ass^{-1}\Delta$ is a torsion pair, and there exists a one-to-one correspondence
\[\xymatrix{\spcl\Delta\ar@<0.5ex>[r]^-{\Phi_\Delta}&\tp(\Ass^{-1}\Delta)\ar@<0.5ex>[l]^-{\Psi_\Delta}}.\]
Here, the mutually inverse bijections $\Phi_\Delta,\Psi_\Delta$ are defined by $\Phi_\Delta(Z)=(\Ass^{-1}Z,\Ass^{-1}(\Delta\setminus Z))$ and $\Psi_\Delta(\T,\F)=\Delta\cap\Supp\T$.
\end{thm}

\begin{proof}
It suffices to show the following three statements.
\begin{enumerate}[\rm(a)]
\item 
The pair $(\Ass^{-1}(\Delta\cap W),\Ass^{-1}(\Delta\setminus W))$ is a torsion pair in $\Ass^{-1}\Delta$ for all specialization-closed subsets $W$ of $\Spec R$.
\item 
$\Psi_\Delta(\Phi_\Delta(\Delta\cap W))=\Delta\cap W$ for all specialization-closed subsets $W$ of $\Spec R$.
\item 
$\Phi_\Delta(\Psi_\Delta(\T,\F))=(\T,\F)$ for all weak torsion pairs $(\T,\F)$ in $\Ass^{-1}\Delta$.
\end{enumerate}

(a) If $T\in\Ass^{-1}(\Delta\cap W)$ and $F\in\Ass^{-1}(\Delta\setminus W)$, then $\Ass(\Hom_R(T,F))=\Supp T\cap\Ass F\subseteq W\cap W^\complement =\emptyset$. Therefore, $\Hom_R(T,F)=0$. For every $R$-module $M\in\Ass^{-1}\Delta$, there is an exact sequence $0\to\Gamma_W(M)\to M\to M/\Gamma_W(M)\to0$. It follows from Lemma \ref{formula gamma ass} that $\Ass(\Gamma_W(M))=\Ass M\cap W\subseteq\Delta\cap W$ and $\Ass(M/\Gamma_W(M))=\Ass M\setminus W\subseteq\Delta\setminus W$. Hence, we have $\Gamma_W(M)\in\Ass^{-1}(\Delta\cap W)$ and $M/\Gamma_W(M)\in\Ass^{-1}(\Delta\setminus W)$. Thus, $(\Ass^{-1}(\Delta\cap W),\Ass^{-1}(\Delta\setminus W))$ is a torsion pair in $\Ass^{-1}\Delta$.

(b) Since $W$ is a specialization-closed subset of $\Spec R$, we see that $\Psi_\Delta(\Phi_\Delta(\Delta\cap W))=\Delta\cap\Supp(\Ass^{-1}(\Delta\cap W))\subseteq\Delta\cap\Supp(\Supp^{-1}W)=\Delta\cap W$. On the other hand, if $\p$ is a prime ideal in $\Delta\cap W$, then $R/\p\in\Ass^{-1}\{\p\}\subseteq\Ass^{-1}(\Delta\cap W)$. Therefore, $\p\in\Delta\cap\Supp(R/\p)\subseteq\Delta\cap\Supp(\Ass^{-1}(\Delta\cap W))=\Psi_\Delta(\Phi_\Delta(\Delta\cap W))$. Thus, $\Delta\cap W\subseteq \Psi_\Delta(\Phi_\Delta(\Delta\cap W))$.

(c) Let $W=\Supp\T$. By Remark \ref{rem fg} and Lemma \ref{formula perp delta}, we see that $(\T,\F)=f_\Delta(\T,\F)=(\Ass^{-1}(\Delta\cap W),\Ass^{-1}(\Delta\setminus W))=\Phi_\Delta(\Delta\cap W)$. Therefore, by (b), we have $\Phi_\Delta(\Psi_\Delta(\T,\F))=\Phi_\Delta(\Psi_\Delta(\Phi_\Delta(\Delta\cap W)))=\Phi_\Delta(\Delta\cap W)=(\T,\F)$.
\end{proof}

The next result determines the form of canonical sequences of torsion pairs in a subcategory of $\mod R$.

\begin{cor}\label{cor cano seq}
Let $\X$ be a subcategory of $\mod R$.
\begin{enumerate}[\rm(1)]
\item
For every weak torsion pair $(\T,\F)$ of $\X$, there exist a specialization-closed subset $W$ of $\Spec R$ such that $(\T,\F)=(\Ass^{-1}W\cap\X,\Ass^{-1}W^\complement \cap\X)$.
\item 
Let $W$ be a specialization-closed subset of $\Spec R$ and $M$ an $R$-module in $\X$. If there exists a short exact sequence $0\to T\to M\to F\to0$ such that $T\in\Ass^{-1}W\cap\X$ and $F\in\Ass^{-1}W^\complement\cap\X$, then the sequence $0\to T\to M\to F\to0$ is isomorphic to the sequence $0\to\Gamma_W(M)\to M\to M/\Gamma_W(M)\to0$. In particular, one has $\Gamma_W(M)\in\T$ and $M/\Gamma_W(M)\in\F$.
\item 
Let $W$ be a specialization-closed subset of $\Spec R$. Then the pair $(\Ass^{-1}W\cap\X,\Ass^{-1}W^\complement \cap\X)$ is a torsion pair in $\X$ if and only if one has $\Gamma_W(M),M/\Gamma_W(M)\in\X$ for all $M\in\X$.
\end{enumerate}
\end{cor}

\begin{proof}
(1) By virtue of Corollary \ref{cor wtp restrict}, there exists a weak torsion pair $(\T',\F')$ in $\mod R$ such that $(\T,\F)=(\T'\cap\X,\F'\cap\X)$. It follows from Theorem \ref{thm tp} that there exist a specialization-closed subset $W$ of $\Spec R$ such that $(\T',\F')=(\Ass^{-1}W,\Ass^{-1}W^\complement )$. Therefore, we have $(\T,\F)=(\Ass^{-1}W\cap\X,\Ass^{-1}W^\complement \cap\X)$.

(2) This is clear from Lemma \ref{lem ex seq torsion}.

(3) This follows immediately from (2).
\end{proof}

We next study torsion pairs of a particular form in $\Ass^{-1}\Delta$. Let $I$ be an ideal of $R$ and $\X$ a subcategory of $\mod R$. We identify the category $\mod R/I$ with the subcategory $\{M\in\mod R\mid IM=0\}$ of $\mod R$ and set $\CC_\X(I)\coloneq\X\cap\mod R/I$. We say that a pair $(\T,\F)$ of subcategories of $\X$ is {\em basic} if there exist ideals $I,J$ of $R$ such that $(\T,\F)=(\CC_\X(I),\CC_\X(J))$. We denote the collection of basic torsion pairs of $\X$ by $\btp(\X)$. We denote by $X_\Delta$ the subset of $\spcl\Delta$ consisting of $C\in\cl\Delta$ such that $C\sqcup C'=\Delta$ for some $C'\in\cl\Delta$. Basic torsion pairs in $\Ass^{-1}\Delta$ are characterized as follows.

\begin{thm}\label{thm btp}
The following conditions are equivalent for all ideals $I,J$ of $R$.
\begin{enumerate}[\rm(1)]
\item 
The pair $(\CC_\Delta(I),\CC_\Delta(J))$ is a torsion pair in $\Ass^{-1}\Delta$.
\item 
The equality $V(I)\cap\V(J)\cap\Delta=\emptyset$ holds, and $IJM=0$ for all objects $M\in\Ass^{-1}\Delta$.
\item 
One has $\CC_\Delta(I)=\Ass^{-1}(\Delta\cap\V(I))$, $\CC_\Delta(J)=\Ass^{-1}(\Delta\cap\V(J))$, and $(\V(I)\cap\Delta)\sqcup(\V(J)\cap\Delta)=\Delta$.
\end{enumerate}
In particular, the restriction of $\Psi_\Delta$ to $\btp(\Ass^{-1}\Delta)$ gives an injective map
$$\btp(\Ass^{-1}\Delta)\longrightarrow X_\Delta.$$
\end{thm}

\begin{proof}
(1)$\Rightarrow$(2):
We first show that $V(I)\cap\V(J)\cap\Delta=\emptyset$. If $\p\in\V(I)\cap\V(J)\cap\Delta$, then it follows from Lemma \ref{lem intersection} that $R/\p\in\CC_\Delta(I)\cap\CC_\Delta(J)=0$. This is a contradiction. Thus, we have $V(I)\cap\V(J)\cap\Delta=\emptyset$.

We next show that $IJM=0$ for all $M\in\Ass^{-1}\Delta$. If $M\in\Ass^{-1}\Delta$, then there exists an exact sequence $0\to T\to M\to F\to0$ with $T\in\CC_\Delta(I)$ and $F\in\CC_\Delta(J)$. As $IT=JF=0$, we obtain $IJM=0$.

(2)$\Rightarrow$(3):
We first prove $\CC_\Delta(I)=\Ass^{-1}(\Delta\cap\V(I))$. Let $M$ be an $R$-module in $\CC_\Delta(I)$. Then we see that $\Ass M\subseteq\Delta$ and $\Supp M\subseteq\V(I)$. Therefore, we have $M\in\Ass^{-1}(\Delta\cap\V(I))$, which proves the inclusion $\CC_\Delta(I)\subseteq\Ass^{-1}(\Delta\cap\V(I))$. Let $M$ be an $R$-module in $\Ass^{-1}(\Delta\cap\V(I))$. Then we have $\Ass M\cap\V(J)\subseteq\V(I)\cap\V(J)\cap\Delta=\emptyset$. Thus, we obtain $\Gamma_J(M)=0$, and therefore $\grade(J,M)>0$. Let $x\in J$ be a regular element of $M$. Since $xIM\subseteq IJM=0$, we see that $IM=0$. Thus, we have $M\in\CC_\Delta(I)$ and the inclusion $\Ass^{-1}(\Delta\cap\V(I))\subseteq\CC_\Delta(I)$ follows. Hence, we conclude that $\CC_\Delta(I)=\Ass^{-1}(\Delta\cap\V(I))$. Similarly, the equality $\CC_\Delta(J)=\Ass^{-1}(\Delta\cap\V(J))$ holds.

For the equality $(\V(I)\cap\Delta)\sqcup(\V(J)\cap\Delta)=\Delta$, it suffices to show that $\Delta\subseteq\V(I)\cup\V(J)$. For every prime ideal $\p\in\Delta$, we have $IJ(R/\p)=0$ since $R/\p\in\Ass^{-1}\Delta$. It follows that $\p\in\V(IJ)=\V(I)\cup\V(J)$. Therefore, we have $\Delta\subseteq\V(I)\cup\V(J)$.

(3)$\Rightarrow$(1):
This implication is evident from Theorem \ref{thm tp}.
\end{proof}

The above theorem shows that for a basic torsion pair $(\T,\F)$ in $\Ass^{-1}\Delta$, the pair $(\F,\T)$ is also a basic torsion pair in $\Ass^{-1}\Delta$. Therefore, Proposition \ref{prop cano seq split} yields the following corollary since $\Ass^{-1}\Delta$ is a Serre subcategory of $\mod R$ if $\Delta$ is specialization-closed subset of $\Spec R$.

\begin{cor}\label{cor Delta spcl}
Suppose that $\Delta$ is a specialization-closed subset of $\Spec R$. Then, for every basic torsion pair $(\T,\F)$, all canonical sequences with respect to $(\T,\F)$ split. In particular, all indecomposable $R$-modules in $\Ass^{-1}\Delta$ belong to either $\T$ or $\F$.
\end{cor}

The following theorem shows that the map in Theorem \ref{thm btp} is bijective if $\Delta$ contains $\Min R$.

\begin{thm}\label{thm btp min}
Suppose that $\Delta$ contains $\Min R$. Then the following conditions are equivalent for all pairs $(\T,\F)$ of subcategories $\Ass^{-1}\Delta$.
\begin{enumerate}[\rm(1)]
\item 
The pair $(\T,\F)$ is a basic torsion pair in $\Ass^{-1}\Delta$.
\item 
There exist ideals $I,J$ of $R$ such that $(\T,\F)=(\CC_\Delta(I),\CC_\Delta(J))$,  $\V(I)\cap\V(J)\cap\Delta=\emptyset$, and $IJ=0$.
\item 
There exists an element $C$ of $X_\Delta$ such that $(\T,\F)=\Phi_\Delta(C)$.
\end{enumerate}
In particular, there exists a one-to-one correspondence
\[\xymatrix{X_\Delta\ar@<0.5ex>[r]&\btp(\Ass^{-1}\Delta)\ar@<0.5ex>[l]}\]
which is the restriction of the correspondence in Theorem \ref{thm tp}.
\end{thm}

\begin{proof}
The implications (2)$\Rightarrow$(1) and (1)$\Rightarrow$(3) are clear by Theorem \ref{thm btp}. We prove (3)$\Rightarrow$(2).

Assume that there exists an element $C$ of $X_\Delta$ such that $(\T,\F)=\Phi_\Delta(C)$. Then there exist ideals $K,L$ of $R$ such that $C=\V(K)\cap\Delta$ and $(\V(K)\cap\Delta)\sqcup(\V(L)\cap\Delta)=\Delta$. Since $\Min R\subseteq\Delta\subseteq\V(K)\cup\V(L)=\V(KL)$, we obtain $\V(KL)=\Spec R=V(0)$. Therefore, there exists a positive integer $n$ such that $K^nL^n=0$. Setting $I\coloneq K^n$ and $J\coloneq L^n$, we have $V(I)\cap\V(J)\cap\Delta=\V(K)\cap\V(L)\cap\Delta=\emptyset$ and $IJ=0$. It follows from Theorem \ref{thm btp} that $(\CC_\Delta(I),\CC_\Delta(J))=\Phi_\Delta(\V(I)\cap\Delta)$. Therefore, we have $(\T,\F)=\Phi_\Delta(\V(K)\cap\Delta)=\Phi_\Delta(\V(I)\cap\Delta)=(\CC_\Delta(I),\CC_\Delta(J))$. This proves (3)$\Rightarrow$(2).
\end{proof}

Combining Theorems \ref{thm tp}, \ref{thm btp}, and \ref{thm btp min}, we obtain Theorem \ref{1}.

Next, we consider weak torsion pairs in arbitrary subcategories of $\mod R$. Our next main aim is to prove Theorem \ref{3}. We establish a correspondence between $\spcl\Delta$ and $\wtp(\X)$ for a subcategory $\X$ of $\Ass^{-1}\Delta$ using the maps we defined earlier.

\begin{dfn}
Let $\X$ be a subcategory of $\Ass^{-1}\Delta$.
\begin{enumerate}[\rm(1)]
\item
We denote the map $(f^\X_\Delta|_{\wtp(\Ass^{-1}\Delta)})\circ\Phi_\Delta:\spcl\Delta\to\wtp(\X)$ by $\Phi^\X_\Delta$.
\item 
We denote the map $\Psi_\Delta\circ(f_\Delta|_{\wtp(\X)}):\wtp(\X)\to\spcl\Delta$ by $\Psi^\X_\Delta$.
\end{enumerate}
\end{dfn}

Note that if $\X=\Ass^{-1}\Delta$, then $\Phi^\X_\Delta$ (resp. $\Psi^\X_\Delta$) coincides with $\Phi_\Delta$ (resp. $\Psi_\Delta$). We provide some important properties of $\Phi^\X_\Delta$ and $\Psi^\X_\Delta$.

\begin{prop}\label{prop wtp X}
Let $\X$ be a subcategory of $\Ass^{-1}\Delta$.
\begin{enumerate}[\rm(1)]
\item 
The equality $\Phi^\X_\Delta\circ\Psi^\X_\Delta=\id_{\wtp(\X)}$ holds. In particular, $\Phi^\X_\Delta$ is surjective and $\Psi^\X_\Delta$ is injective.
\item 
Let $Z$ be specialization-closed subset of $\Delta$. Then $\Phi_\Delta(Z)\cap\X$ is a weak torsion pair in $\X$ if and only if $\Phi^\X_\Delta(Z)=\Phi_\Delta(Z)\cap\X$.
\item
Let $(\T,\F)$ be a weak torsion pair in $\X$. Then one has $\Psi^\X_\Delta(\T,\F)=\Delta\cap\Supp\T$.
\end{enumerate}
\end{prop}

\begin{proof}
(1) By Proposition \ref{prop wtp subcat} and Theorem \ref{thm tp}, we have $\Phi^\X_\Delta\circ\Psi^\X_\Delta=(f^\X_\Delta|_{\wtp(\Ass^{-1}\Delta)})\circ\Phi_\Delta\circ\Psi_\Delta\circ(f_\Delta|_{\wtp(\X)})=(f^\X_\Delta|_{\wtp(\Ass^{-1}\Delta)})\circ(f_\Delta|_{\wtp(\X)})=\id_{\wtp(\X)}$.

(2) This follows immediately from Remark \ref{rem fg}.

(3) Let $(\T,\F)$ be a weak torsion pair in $\X$. 
Lemma \ref{formula perp delta} and Theorem \ref{thm tp} show that $\Psi_\Delta(f_\Delta(\T,\F))=\Psi_\Delta(\Ass^{-1}(\Delta\cap\Supp\T),\Ass^{-1}(\Delta\setminus\Supp\T))=\Psi_\Delta(\Phi_\Delta(\Delta\cap\Supp\T))=\Delta\cap\Supp\T=\Psi^\X_\Delta(\T,\F)$.
\end{proof}

The above proposition shows that every subcategory of $\mod R$ which has only one associated prime ideal admits only trivial torsion pairs.

\begin{cor}\label{cor singleton}
Let $\X$ be a subcategory of $\mod R$ containing the zero object $0$ such that $\Ass\X$ is a singleton. Then one has $\wtp(\X)=\{(0,\X),(\X,0)\}$.
\end{cor}

\begin{proof}
Let $\Ass\X=\{\p\}$. Then $\X$ is a subcategory of $\Ass^{-1}\{\p\}$. We see that $\spcl\{\p\}=\{\emptyset,\{\p\}\}$, $\Phi_{\{\p\}}(\emptyset)\cap\X=(0,\Ass^{-1}\{\p\})\cap\X=(0,\X)$, and $\Phi_{\{\p\}}(\{\p\})=(\Ass^{-1}\{\p\},0)\cap\X=(\X,0)$. Therefore, (1) and (3) of Proposition \ref{prop wtp X} imply that $\wtp(\X)=\{(0,\X),(\X,0)\}$.
\end{proof}

We next investigate when $\Phi^\X_{\Ass\X}$ and $\Psi^\X_{\Ass\X}$ are bijective for a subcategory $\X$ of $\mod R$. We consider the following condition.

\begin{dfn}
Let $\X$ be a subcategory of $\mod R$. We say that $\X$ satisfies the condition $(\AA)$ if for all prime ideals $\p\in\Ass\X$, there exists an $R$-module $X\in\X$ such that $\Ass X=\{\p\}$.
\end{dfn}

We close this section with the following theorem, which shows that the condition $(\AA)$ is sufficient for $\Phi^\X_{\Ass\X}$ and $\Psi^\X_{\Ass\X}$ to be bijective.

\begin{thm}\label{thm A}
Let $\X$ be a subcategory of $\mod R$ satisfying the condition $(\AA)$. Then one has $\Phi^\X_{\Ass\X}(Z)=\Phi_{\Ass\X}(Z)\cap\X$ for all specialization-closed subsets $Z$ of $\Ass\X$ and the maps $\Phi^\X_{\Ass\X},\Psi^\X_{\Ass\X}$ are mutually inverse bijections. Therefore, there exists a one-to-one correspondence
\[\xymatrix{\spcl(\Ass\X)\ar@<0.5ex>[r]^-{\Phi^\X_{\Ass\X}}&\wtp(\X)\ar@<0.5ex>[l]^-{\Psi^\X_{\Ass\X}}}.\]
\end{thm}

\begin{proof}
Let $\Delta=\Ass\X$ and let $Z$ be a specialization-closed subset of $\Delta$. Set $(\T,\F)\coloneq\Phi_\Delta(Z)\cap\X=(\Ass^{-1}Z\cap\X,\Ass^{-1}(\Delta\setminus Z)\cap\X)$. By Proposition \ref{prop wtp X}, it suffices to show that $(\T,\F)$ is a weak torsion pair in $\X$ and $\Psi^\X_{\Delta}(\T,\F)=Z$.

We first claim that $\Delta\cap\Supp\T=Z$. Since $\T\subseteq\Ass^{-1}Z$, we have $\Delta\cap\Supp\T\subseteq\Delta\cap\Supp(\Ass^{-1}Z)=\Psi_\Delta(\Phi_\Delta(Z))=Z$. If $\p$ is a prime ideal in $Z$, then there exists an $R$-module $X\in\X$ such that $\Ass X=\{\p\}$. We see that $X\in\Ass^{-1}Z\cap\X=\T$, and therefore $\p\in Z\cap\Supp X\subseteq\Delta\cap\Supp\T$. This proves the equality $\Delta\cap\Supp\T=Z$. By Lemma \ref{formula perp delta} (1) and the above equality, we have $\T^\perp_\X=\T^\perp_\Delta\cap\X=\Ass^{-1}(\Delta\setminus\Supp\T)\cap\X=\Ass^{-1}(\Delta\setminus Z)\cap\X=\F$.

We next claim that $^\perp_\X\F=\T$. The inclusion $\T\subseteq{}^\perp_\X\F$ is clear. Let $M$ be an $R$-module in $^\perp_\X\F$. If $\p$ is a prime ideal in $\Delta\setminus Z$, then there exists an $R$-module $X\in\X$ such that $\Ass X=\{\p\}$. Since $X\in\F$, we have $\Hom_R(M,X)=0$, which implies that $\Ass M\cap\{\p\}=\Ass M\cap\Ass X=\emptyset$. It follows that $\Ass M\cap(\Delta\setminus Z)=\emptyset$. Therefore, we have $M\in\Ass^{-1}Z\cap\X=\T$. This proves the equality $^\perp_\X\F=\T$.

Hence, we conclude that $(\T,\F)$ is a torsion pair in $\X$ and $\Psi^\X_{\Delta}(\T,\F)=\Delta\cap\Supp\T=Z$. This completes the proof of the theorem.
\end{proof}

\section{Applications}
In this section, we give applications of the results from Section 4. Throughout this section, let $R$ be a ring. We begin with a classification of (basic) torsion pairs in $\mod R$.

\begin{cor}\label{cor tp modR}
The equality
$$\btp(\mod R)=\{(\mod R/I,\mod R/J)\mid\text{$I,J$ are ideals of $R$ such that $IJ=0$ and $I+J=R$}\}$$
holds and there is a commutative diagram
$$\xymatrix{\spcl(\Spec R)\ar@<0.5ex>[r]^-{\Phi_{\Spec R}}&\tp(\mod R)\ar@<0.5ex>[l]^-{\Psi_{\Spec R}}\ar@{=}[r]&\wtp(\mod R)\\
X_{\Spec R}\ar[u]_\inc\ar@<0.5ex>[r]&\btp(\mod R)\ar[u]_\inc\ar@<0.5ex>[l]&}$$
where the horizontal maps are mutually inverse bijections. In particular, if $R$ is local or $\Min R$ is a singleton, then $\btp(\mod R)=\{(0,\mod R),(\mod R,0)\}$.
\end{cor}

\begin{proof}
The first equality and the commutative diagram follow from Theorems \ref{thm tp} and \ref{thm btp min} since $\mod R=\Ass^{-1}(\Spec R)$. If $R$ is local or $\Min R$ is a singleton, then we have $X_{\Spec R}=\{\emptyset,\Spec R\}$, which yields the final assertion.
\end{proof}

\begin{rem}
The assertion in the last sentence of Corollary \ref{cor tp modR} also follows from Proposition \ref{prop cano seq split}. Indeed, suppose that $R$ is local or $\Min R$ is a singleton, and $(\T,\F)$ is a basic torsion pair in $\mod R$. Since the $R$-module $R$ is indecomposable, $R$ belongs to either $\T$ or $\F$ by Corollary \ref{cor Delta spcl}. We may assume $R\in\T$. As $\T$ is a Serre subcategory of $\mod R$, every quotient of a free module belongs to $\T$. Therefore, $\T=\mod R$ and $\F=0$.
\end{rem}

When $(R,\m)$ is a local ring, we set $\Spec_0R\coloneq\Spec R\setminus\{\m\}$ and denote by $\dep R$ the subcategory of $\mod R$ consisting of $R$-modules of positive depth. We note that $\dep R=\Ass^{-1}(\Spec_0R)$. The following is a classification of torsion pairs in $\dep R$.

\begin{cor}\label{cor Spec0}
Suppose that $(R,\m)$ is a local ring with $\dim R>0$. Then one has
\[\btp(\dep R)=\{(\dep R/I,\dep R/J)\mid\text{$I,J$ are ideals of $R$ such that $IJ=0$ and $\V(I)\cap\V(J)\subseteq\{\m\}$}\}\]
and there is a commutative diagram
\[\xymatrix{\spcl(\Spec_0R)\ar@<0.5ex>[r]^-{\Phi_{\Spec_0R}}&\tp(\dep R)\ar@<0.5ex>[l]^-{\Psi_{\Spec_0R}}\ar@{=}[r]&\wtp(\dep R)\\
X_{\Spec_0R}\ar[u]_\inc\ar@<0.5ex>[r]&\btp(\dep R)\ar[u]_\inc\ar@<0.5ex>[l]}\]
where the maps in each row are mutually inverse bijections. Moreover, if $\depth R\ge2$, then $\btp(\dep R)=\{(0,\dep R),(\dep R,0)\}$.
\end{cor}

\begin{proof}
It follows from $\dim R>0$ that $\Min R\subseteq\Spec_0 R$. For every ideal $I$ of $R$, we see that $\CC_{\Spec_0R}(I)=\dep R\cap\mod R/I=\dep R/I$. Therefore, the assertion follows from Theorems \ref{thm tp} and \ref{thm btp min}. We next suppose that $\depth R\ge2$. Then $\Spec_0R$ is connected by \cite[Proposition 2.1]{H}. Hence, $X_{\Spec_0R}=\{\emptyset,\Spec_0R\}$ and it follows that $\btp(\dep R)=\{(0,\dep R),(\dep R,0)\}$.
\end{proof}

We denote by $\S_1(R)$ the subcategory of $\mod R$ consisting of $R$-modules $M$ satisfying the Serre condition $(\S_1)$, that is, $\depth_{R_\p}M_\p\ge\min\{1,\dim R_\p\}$ for all prime ideals $\p$ of $R$. We note that $\S_1(R)=\Ass^{-1}(\Min R)$. Therefore, torsion pairs in $\S_1(R)$ are classified as follows.

\begin{cor}\label{cor S1}
One has $\wtp(\S_1(R))=\tp(\S_1(R))=\btp(\S_1(R))$ and there is a one-to-one correspondence
\[\xymatrix{2^{\Min R}\ar@<0.5ex>[r]^-{\Phi_{\Min R}}&\tp(\S_1(R))\ar@<0.5ex>[l]^-{\Psi_{\Min R}}}.\]
\end{cor}

\begin{proof}
Since $\Min R$ is discrete, we see that $X_\Delta=\spcl(\Min R)=2^{\Min R}$. Therefore, the assertion follows from Theorems \ref{thm tp} and \ref{thm btp min}.
\end{proof}

We next consider torsion pairs in $\CM(R)$, where $R$ is a Cohen--Macaulay local ring. Note that $\Ass(\CM(R))=\Min R$, and therefore $\CM(R)$ is a subcategory of $\S_1(R)$. We say that $R$ is of {\em finite CM type} if there are only finitely many isomorphism classes of indecomposable maximal Cohen--Macaulay modules. The following corollary describes the fundamental properties of torsion pairs in $\CM(R)$.

\begin{cor}\label{cor CM}
Suppose that $R$ is a Cohen--Macaulay local ring.
\begin{enumerate}[\rm(1)]
\item
All weak torsion pairs in $\CM(R)$ are basic.
\item 
If $\dim R\le 1$, then one has $\wtp(\CM(R))=\tp(\CM(R))$ and there exists a one-to-one correspondence
$$\xymatrix{2^{\Min R}\ar@<0.5ex>[r]^-{\Phi_{\Min R}}&\tp(\CM(R))\ar@<0.5ex>[l]^-{\Psi_{\Min R}}}.$$
\item 
If $\Min R$ is a singleton, then one has $\wtp(\CM(R))=\{(0,\CM(R)),(\CM(R),0)\}$. In particular, if $R$ is of finite CM type with $\dim R\ge 2$, then the same equality holds.
\end{enumerate}
\end{cor}

\begin{proof}
(1) Let $(\T,\F)$ be a weak torsion pair in $\CM(R)$. By Corollary \ref{cor wtp restrict}, there exists a weak torsion pair $(\T',\F')$ in $\S_1(R)$ such that $(\T,\F)=(\T',\F')\cap\CM(R)$. It follows from Corollary \ref{cor S1} that $(\T',\F')$ is basic. Therefore, $(\T,\F)$ is basic.

(2) This follows from Corollary \ref{cor S1} since $\CM(R)=\S_1(R)$.

(3) The first assertion is clear from Corollary \ref{cor singleton} since $\Ass(\CM(R))=\Min R$. Suppose that $R$ is of finite CM type with $\dim R\ge 2$. By \cite[Theorem 7.12]{LW}, $R$ is an isolated singularity; that is, $R_\p$ is regular for all nonmaximal prime ideals $\p$. In particular, $R$ is a normal domain. Hence, we have $\Min R=\{0\}$.
\end{proof}

The following corollary shows that torsion pairs in $\CM(R)$ can be explicitly described if $R$ is a one-dimensional reduced local ring.

\begin{cor}
If $R$ is a one-dimensional reduced local ring, then one has 
$$\tp(\CM(R))=\{(\CM(R/I),\CM(R/J))\mid\Min R=C\sqcup C', I=\bigcap_{\p\in C}\p,J=\bigcap_{\p\in C'}\p\}.$$
\end{cor}

\begin{proof}
Since $R$ is Cohen--Macaulay, Corollary \ref{cor CM} (2) implies $\tp(\CM(R))=\{\Phi_{\Min R}(C)\mid C\in\Min R\}$. Let $C$ be a subset of $\Min R$. Set $C'=\Min R\setminus C$, $I=\bigcap_{\p\in C}\p$, and $J=\bigcap_{\p\in C'}\p$. Then we see that $C=\V(I)\cap\Min R$, $C'=\V(J)\cap\Min R$, and $I\cap J=0$. Therefore, $\V(I)\cap\V(J)\cap\Min R=C\cap C'=\emptyset$ and $IJ=0$. By Theorem \ref{thm btp}, we have $\Phi_{\Min R}(C)=\Phi_{\Min R}(\V(I)\cap\Min R)=(\CC_{\Min R}(I),\CC_{\Min R}(J))$. If $C=\Min R$, then $I=R$, and hence $\CC_{\Min R}(I)=0=\CM(R/I)$. If $C\ne\Min R$, then $R/I$ is Cohen--Macaulay and $\CC_{\Min R}(I)=\CM(R/I)$ since $R/I$ is reduced and $\dim (R/I)=1$. Thus $\Phi_{\Min R}(C)=(\CM(R/I),\CM(R/J))$. This implies $\tp(\CM(R))=\{\Phi_{\Min R}(C)\mid C\in\Min R\}=\{(\CM(R/I),\CM(R/J))\mid\Min R=C\sqcup C', I=\bigcap_{\p\in C}\p,J=\bigcap_{\p\in C'}\p\}$.
\end{proof}

In the following example, we see that a canonical torsion pair arises over a fiber product of one-dimensional Cohen--Macaulay local rings.

\begin{ex}
Let $R$ and $S$ be one-dimensional Cohen--Macaulay local rings and $Q$ an artinian local ring. Suppose that there are surjective ring homomorphisms $R\overset{f}\to Q$, $S\overset{g}\to Q$. Consider the fiber product
\[\xymatrix{
R\underset{Q}\times S\ar[r]^-\beta\ar[d]_\alpha&S\ar[d]^g\\
R\ar[r]^-f&Q
.}\]
We set $T\coloneq R\underset{Q}\times S$, $I\coloneq\Ker\alpha$, and $J\coloneq\Ker\beta$. It follows from \cite[Lemmas 1.2 and 1.5]{AAM} that $T$ is a one-dimensional Cohen--Macaulay local ring. We see that $I\cap J=0$ and that $Q\cong T/(I+J)$. Therefore, we have $IJ=0$ and $\V(I)\cap\V(J)=\V(I+J)=\{\m\}$ where $\m$ is the maximal ideal of $T$. Therefore, by Corollary \ref{cor Spec0}, the pair $(\dep(T/I),\dep(T/J))=(\CM(R),\CM(S))$ is a basic torsion pair in $\CM(T)$.
\end{ex}

If $R$ is a local hypersurface ring, then we can apply Theorem \ref{thm A} to classify weak torsion pairs in $\CM(R)$ as follows.

\begin{cor}\label{cor HS}
Let $R$ be a hypersurface local ring $S/(f)$ where $S$ is a regular local ring and $f$ is a nonunit element of $S$. Then one has $\Phi^{\CM(R)}_{\Min R}(Z)=\Phi_{\Min R}(Z)\cap\CM(R)$ for all subsets of $Z$ of $\Min R$ and the maps $\Phi^{\CM(R)}_{\Min R},\Psi^{\CM(R)}_{\Min R}$ are mutually inverse bijections. Therefore, there exists a one-to-one correspondence
\[\xymatrix{2^{\Min R}\ar@<0.5ex>[r]^-{\Phi^{\CM(R)}_{\Min R}}&\wtp(\CM(R))\ar@<0.5ex>[l]^-{\Psi^{\CM(R)}_{\Min R}}}.\]
\end{cor}

\begin{proof}
We show that the subcategory $\CM(R)$ of $\mod R$ satisfies the condition $(\AA)$. The case $f=0$ is clear from Corollary \ref{cor CM} (3). We therefore assume that $f\ne0$. Let $f=p_1\cdots p_n$ be a prime decomposition of $f$. Then $\Ass(\CM(R))=\Ass R=\{\p_1,\dots,\p_n\}$, where $\p_i=(p_i)/(f)$ for $i=1,\dots,n$. We see that the $R$-module $R/\p_i$ is maximal Cohen-Macaulay and that $\Ass_R(R/\p_i)=\{\p_i\}$. Hence $\CM(R)$ satisfies the condition $(\AA)$ and the assertion follows from Theorem \ref{thm A}.
\end{proof}

We end this section by giving an example of a weak torsion pair which is not a torsion pair.

\begin{ex}
Let $k$ be an algebraically closed field, and consider the $2$-dimensional hypersurface local ring $R=k\llbracket x,y,z\rrbracket/(xy)$.
Then $\{R,(x), (y),(x,y^i),(x,z^i)\mid i=1,2,\dots\}$ is the complete list of nonisomorphic indecomposable maximal Cohen--Macaulay $R$-modules; see \cite[Theorem 5.3]{BD} for instance. 
We observe that $\Min R=\{(x),(y)\}$, that $\Ass^{-1}\{(x)\}\cap\CM(R)=\add(y)=\CC_{\CM(R)}((x))$, and that $\Ass^{-1}\{(y)\}\cap\CM(R)=\add (x)=\CC_{\CM(R)}((y))$. 
By Corollary \ref{cor HS}, we have $\wtp(\CM(R))=\{(0,\CM(R)),(\CM(R),0), (\add(x),\add(y)),(\add(y),\add(x))\}$.

However, the pair $(\add(x),\add(y))$ is not a torsion pair in $\CM(R)$. 
Indeed, suppose that there exists a short exact sequence $0\to M\to (y,z)\to N\to0$ in $\CM(R)$ such that $M\in\add(x)$ and $N\in\add(y)$. 
Then, by Corollary \ref{cor cano seq} (2), we have $M\cong\Gamma_{(x)}(y,z)$ and $N\cong(y,z)/\Gamma_{(x)}(y,z)$. 
Since $\Gamma_{(x)}(y,z)=(x)\cap(y,z)$, it follows that $N\cong(y,z)/((x)\cap(y,z))\cong(x,y,z)/(x)$. 
The short exact sequence $0\to(x,y,z)/(x)\to R/(x)\to k\to0$ together with the depth lemma shows that $\depth_R(x,y,z)/(x)=1<2=\dim R$. Thus, $N$ is not maximal Cohen--Macaulay, which is a contradiction. By the same argument, the pair $(\add(y),\add(x))$ is not a torsion pair in $\CM(R)$. Consequently, we have $\tp(\CM(R))=\{(0,\CM(R)),(\CM(R),0)\}$.
\end{ex}

\begin{ac}
The authors thank Yuki Mifune for giving them valuable comments.
\end{ac}


\end{document}